\documentclass{amsart}
\usepackage[utf8]{inputenc}
\usepackage{amsmath,amssymb,amsbsy,amsfonts,amsthm,latexsym,mathabx, xcolor, amsopn,amstext,amsxtra,euscript,amscd,mathrsfs,array,mathtools,enumerate,tikz-cd}
\usepackage{url}
\usepackage[colorlinks,linkcolor=blue,anchorcolor=blue,citecolor=blue,backref=page]{hyperref}
\usepackage[norefs,nocites]{refcheck}
\usepackage{float}
\hypersetup{breaklinks=true}
\usepackage[english]{babel}
\usepackage{enumitem}
\usepackage{a4wide}
\usepackage{enumerate}
\usepackage[noadjust]{cite}

\newtheorem {thm}{Theorem}
\newtheorem {cor}[thm]{Corollary}
\newtheorem {lem}[thm]{Lemma}
\newtheorem {prop}[thm]{Proposition}

\theoremstyle{definition}
\newtheorem {defi}[thm]{Definition}
\newtheorem {rem}[thm]{Remark}
\newtheorem {exa}[thm]{Example}

\newcommand{\ord}{\mathrm{ord}}

\DeclareMathOperator{\Gal}{Gal}

\newcommand{\kbar}{\overline{k}}

\newcommand{\Z}{\mathbb{Z}}
\newcommand{\ch}{\mathrm{char}}
\newcommand{\tk}{\widetilde{k}}
\newcommand{\calO}{\mathcal{O}}
\newcommand{\mpp}{\mathfrak{m}}

\newcommand{\trdeg}{\mathrm{trdeg}}

\newcommand{\Prime}{\mathrm{Prime}}

\title{Kummer theory for function fields}
\author{Félix Baril Boudreau and Antonella Perucca}
\date{}

\begin{document}
	
	\begin{abstract}
		We develop Kummer theory for algebraic function fields in finitely many transcendental variables. We consider any finitely generated  Kummer extension (possibly, over a cyclotomic extension) of an algebraic function field, and describe the structure of its Galois group. Our results show in a precise sense how the questions of computing the degrees of these extensions and of computing the group structures of their Galois groups reduce to the corresponding questions for the Kummer extensions of their constant fields.
	\end{abstract}
	
	
	
	\maketitle
	
	
	\section{Introduction}
	
	Let $n \geq 1$ be an integer and let $K$ be a field containing a primitive $n$-th root of unity. A Kummer extension of $K$ is a field extension obtained by adjoining a finite number of $n$-th roots of elements of $K$. This extension is Abelian of exponent dividing $n$. Conversely, Kummer theory says that any finite Abelian extension of $K$ of exponent dividing $n$ can be obtained by taking the $n$-th roots of finitely many elements of $K^\times$, the multiplicative subgroup of $K$. Moreover, Abelian extensions of $K$ of exponent dividing $n$ can then be seen to correspond bijectively with subgroups of the quotient group $K^\times/K^{\times n}$. In this work, we develop Kummer theory for arbitrary algebraic function fields and go beyond the classical theory: Using parameters that measure the divisibility of elements up to constants, we completely solve the problem of understanding the nature of Kummer theory over algebraic function fields by showing in a precise way that this theory reduces to the Kummer theory over the corresponding constant fields.
	
	We consider a function field $K/k$ and suppose without loss of generality that $k$ is algebraically closed in $K$. For any element
	$\alpha \in K^\times\setminus k^\times$, and any prime number $\ell\neq \ch(k)$, we can speak of its \emph{$\ell$-divisibility parameter modulo constants over $K$}. This is the largest non-negative integer $D$ such that $\alpha\in k K^{\times \ell^D}$.
	
	For an integer $M \geq 1$ not divisible by $\ch(k)$, and a choice $\zeta_M$ of primitive $M$-th root of unity in a fixed algebraic closure $\overline{K}$ of $K$, the field $K(\zeta_M)$ is a cyclotomic extension of $K$. Given a finitely generated subgroup $G$ of $K^\times$, we consider the Kummer extensions
	\begin{equation}\label{Kummer}
		K(\zeta_{M}, \sqrt[N]{G})/K(\zeta_{M})\qquad \text{where $N\mid M$}\,.
	\end{equation}
	In this article we show that, to compute the degree and the structure of the Galois group of the above Kummer extensions, we may reduce to computations involving only the constant field $k$.
	
	By Kummer theory, we may suppose that $N=\ell^n$ for some prime number $\ell$ and some integer $n\geq 1$. 
	We can write $G=G_0G'$ where $G_{0}:=G\cap k^\times$ and where  $G'$ is a subgroup of $G$ such that $G'\cap k^\times = \{1\}$. We denote by $r$ the rank of $G'$ and we suppose that $r$ is strictly positive (for $r=0$, see Remark \ref{zero}). A careful choice of basis for $G'$, which we call $\ell$-\textit{good basis modulo constants} (see Definition \ref{good-basis-modulo}), allows us to conduct a finer analysis of Galois group structures of Kummer extensions. The following result, which combines Corollary \ref{cor-Kummerdeg-modulo} and Theorem \ref{exists}, considers the so-called \textit{geometric Kummer extension} $K\bar{k}(\sqrt[\ell^n]{G})/K\bar{k}$:
	
	\begin{thm}
		There is a basis $\{\alpha_1,\ldots, \alpha_r\}$ of $G'$ such that, calling $D_i$ the $\ell$-divisibility parameter modulo constants of $\alpha_i$ over $K(\zeta_\ell)$, the integer $\sum_{i=1}^n D_i$ is maximal (by varying the basis of $G'$). The multiset of the $D_i$'s is uniquely determined (it only depends on $K/k$ and $G$), and for every $n\geq 1$ we have
		$$\Gal(K\bar{k}(\sqrt[\ell^n]{G})/K\bar{k})\simeq \prod_{i=1}^r \mathbb Z/\ell^{\max(n-D_i, 0)} \mathbb Z\,.$$
	\end{thm}
	
	To simplify the notation, suppose that $\zeta_M \in k$. We deduce the following result on the cardinality of the Kummer extensions (see Theorem \ref{Main_Theorem_Part_2}):
	
	\begin{thm}
		With the above notation, let $c_i\in k$ be such that $\alpha_i\equiv c_i \bmod K^{\times \ell^{D_i}}$. Then we have
		$$[K(\sqrt[\ell^{n}]{G}):K]=[k,\sqrt[\ell^{n}]{G_0}, \sqrt[\ell^{\min(n, D_1)}]{c_1}, \ldots, \sqrt[\ell^{\min(n, D_r)}]{c_r}): k]\cdot \prod_{i=1}^r \ell^{\max(n-D_i,0)}\,.$$
	\end{thm}
	
	Up to relabelling, we suppose without loss of generality that $D_1\leq D_2\leq \cdots \leq D_r$. Our main result gives the precise Galois group structure of the Kummer extension $K(\sqrt[\ell^n]{G})/K$ (see Theorem \ref{special}): 
	
	\begin{thm}
		With the above notation, let $d_i\in \mathbb Z_{\geq 0}\cup \{\infty\}$ be maximal such that the class of $c_i$ is an $\ell^{d_i}$-th power in $k^{\times}/\langle G_0, c_j : j>i\rangle$. Then there is a group isomorphism
		$$\Gal\left(K\left( \sqrt[\ell^{n}]{G}\right)/K\right)\simeq \Gal\left(k\left(\sqrt[\ell^{n}]{G_0}\right)/k\right)\times  \prod_{i=1}^r \mathbb Z/ \ell^{\max (n-\min(D_i, d_i),0)} \mathbb Z\,.
		$$
	\end{thm}
	
	Thanks to the above results, to determine the cardinality and the structure of the Galois group of the Kummer extensions in \eqref{Kummer}, it suffices to compute the parameters $D_i$'s and the constant elements $c_i$'s, and this then reduces to computations involving only the constant field $k$. Such computations can be done in principle for finite fields, but there is also an explicit finite procedure to perform them for number fields and $p$-adic fields (see \cite{ACPPP, padic}).\\
	
	The structure of the paper is as follows: Section \ref{sec:pre} contains background theory on algebraic function fields and a lemma about rank $1$ discrete valuations and $\ell$-divisibility modulo constants of elements. Sections \ref{sec:ext} and \ref{sec:ex-app} are devoted to studying extensions of function fields (with some general results that can be of independent interest). The theory of divisibility modulo constants is developed in Section \ref{sec:div}, and our main results for Kummer theory are proven in Section \ref{sec:Kum}.
	
	The idea of using parameters to consider the divisibility of elements in a given field comes from past works of the second-named author on Kummer theory for number fields. Nevertheless, we consider a completely new setting and we provide several auxiliary results to develop Kummer theory for function fields in full generality.

	
	\section{Prerequisites on function fields and notation}\label{sec:pre}
	
	An \textit{algebraic function field (or function field, for short)} is a triple $(K,k,n)$, comprised of a field $K$ containing a field $k$ such that $K$ is finitely generated over $k$ of transcendence degree $\trdeg(K/k)$ equal to some positive integer $n$. In this case, we write $K/k$ (and can refer to it as a \textit{function field in $n$ variables}). Let $\{t_1, t_2, \cdots, t_n \}$ be a transcendence basis for $K$ over $k$. The function field $K/k$ is a \textit{rational function field} if $K = k(t_1,t_2,\cdots,t_n)$ and it is a \textit{global function field} if $n = 1$ and $k$ is a finite field.
	
	In this manuscript, a \textit{valuation on} $K/k$ is a surjective map $v$ from $K$ to $\Z \cup \{\infty\}$, with $\Z$ the additive group of integers and $\infty$ a symbol subject to the rules $n + \infty = \infty + n = \infty + \infty = \infty$ for all $n \in \Z$. The map $v$ is satisfying for all $f_1, f_2 \in K^\times$ the properties $v(f_1 f_2) = v(f_1) + v(f_2)$ and $v(f_1 + f_2) \geq \min\{ v(f_1), v(f_2)\}$, is such that if $f \in k^\times$, then $v(f) = 0$ and $v(0) = \infty$. In particular, our valuations are \textit{discrete} in the usual terminology.
	
	The ring $\calO_v$ of all $f \in K^\times$ satisfying $v(f) \geq 0$, is called the \textit{valuation ring} of $v$. It is a local ring whose maximal ideal $\mpp_v$ is the set of $f \in \calO_v$ satisfying $v(f) > 0$. The field $k_v := \calO_v/\mpp_v$ is called the \textit{residue field} of $v$. It can be seen as a field extension of $k$ and the transcendence degree $\trdeg(k_v/k)$ of $k_v$ over $k$ will be the \textit{dimension of} $v$ \textit{over} $k$. A \textit{prime divisor of} $K/k$ is a valuation of $K/k$ of dimension $n-1$ over $k$. Each prime divisor is a discrete valuation of \textit{rank} $1$ (that is, $\calO_v$ has Krull dimension $1$) and $k_v$ is a function field of transcendence degree $n-1$ over $k$. We denote by $\Prime(K/k)$ the set of prime divisors of $K/k$.
	
	The \textit{constant field} of $K/k$ is the algebraic closure $\tk$ of $k$ inside $K$. The extension $\tk/k$ is finite as
	$$
	[\tk : k] = [\tk(t_1,\cdots,t_n) : k(t_1,\cdots, t_n)] \leq [K : k(t_1,\cdots, t_n)]
	$$
	is finite. Therefore, $K/\tk$ is an algebraic function field of $n$ variables over $\tk$.
	
	Fixing an algebraic closure $\overline{K}$ of $K$, if $M$ is a positive integer not divisible by $\ch(k)$, we denote by $\zeta_M$ a primitive root of unity in $\overline{K}$ (or inside the algebraic closure $\kbar$ of $k$ in $\overline{K}$) of order $M$ and we obtain a cyclotomic extension $k(\zeta_M)/k$.
	
	Throughout the text, $\ell$ is a prime number different from $\ch(k)$. We  define $k(\zeta_{\ell^\infty})/k$ inside $\overline{K}$ as the union, over the integers $n\geq 1$, of the cyclotomic extensions $k(\zeta_{\ell^n})/k$. 
	We also write $v_\ell$ for the $\ell$-adic valuation on $\mathbb{Q}^\times$. Also, if $G$ is a group, then $G = \langle \alpha_1, \cdots, \alpha_r \rangle$ means that $G$ is generated by a finite set of elements $\{\alpha_1,\cdots,\alpha_r\}$.
	
	\begin{lem}\label{Lemma_Prime_Divisors}
		Let $f\in K^\times\setminus \tk$ and $c\in \tk^\times$.
		\begin{enumerate}
			\item[(i)] 
			For all $v \in \Prime(K/k)$ we have $v(c) = 0$ and $v(cf) = v(f)$.
			\item[(ii)] There is a finite non-empty set $S$ of $\Prime(K/k)$ such that $v(f) \neq 0$ for all $v \in S$ and $v(f) = 0$ for all $v \notin S$. More precisely, $S$ has at least two elements, $v_1, v_2$ respectively satisfying $\ord_{v_1}(f) > 0$ and $\ord_{v_2}(f) < 0$.
			\item[(iii)] There are only finitely many positive integers $n$ such that $f\in \tk K^{\times n}$. In particular, there is no prime number $\ell$ such that $f$ is infinitely $\ell$-divisible (not even up to constant elements).
		\end{enumerate}
	\end{lem}
	
	\begin{proof}
		\begin{enumerate}
			\item[(i)] This follows immediately from the fact that $v$ is also a valuation on $K/\tk$.
			\item[(ii)] The finiteness of $S$ is classical divisor theory and the rest is proven in \cite[p.99, Ch VI, \S14 and p.175, Ch VII, \S 4 bis]{Zariski_Samuel}.
			\item[(iii)] This follows from (i) and (ii) since if $c f \in \tk K^{\times n}$, then $n$ divides $v(cf) = v(f)$ for all $v \in \Prime(K/k)$.
		\end{enumerate}   
	\end{proof}
	
	
	\section{Extensions of function fields}\label{sec:ext}
	
	
	\subsection{General results}
	The following results hold for fields that are not necessarily function fields.
	
	\begin{defi}
		Let $K$ and $L$ be two extensions of a field $k$ that are contained in some algebraically closed field. The field $K$ is said to be \textit{linearly disjoint from $L$ over $k$} if every finite set of elements of $K$ that is linearly independent over $k$ is linearly independent over $L$. 
	\end{defi}
	
	Remark that if $K$ and $L$ are linearly disjoint over $k$, then $K$ and $L'$ are linearly disjoint over $k$ for any intermediate extension $k \subseteq L' \subseteq L$.
	
	\begin{lem}\label{Lemma_Linearly_Disjoint_Algebraic_Element}
		Let $k$ be a field that is algebraically closed in some field $K$. Let $\alpha$ be an element of a fixed algebraic closure $\kbar$ over $k$. Then $k(\alpha)$ and $K$ are linearly disjoint over $k$, or equivalently, $k(\alpha) \otimes_k K$ is a field, and we have $[k(\alpha):k] = [K(\alpha):K]$. In particular, the minimal polynomial of $\alpha$ over $K$ and over $k$ are the same.
	\end{lem}
	\begin{proof}
		Let $m(x)$ be the minimal polynomial of $\alpha$ over $k$. Any constant factor of $m(x)$ over $K$ has its coefficients algebraic over $k$, hence in $k$. Therefore, $m(x)$ is irreducible over $K$ and the statement follows.
	\end{proof}
	
	\begin{lem}[{\cite[p.253, Proposition 8.4.1]{VillaSalvador_2006}}]\label{Lemma_Algebraically_Closed_Preserved_After_Adding_Variables}
		If a field $k$ is algebraically closed in $K$  and $\{T_i\}_{i \in I}$ is an algebraically independent set over $K$, then $k(\{T_i\}_{i \in I})$ is algebraically closed in $K(\{T_i\}_{i \in I})$.
	\end{lem}
	
	\begin{lem}\label{algebraic}
		Let $K/k$ be a field extension such that $k$ is algebraically closed in $K$. Fix some algebraic closure $\overline{K}$ of $K$ and fix an algebraic closure $\kbar$ of $k$ contained in $\overline{K}$.
		Let $F/K$ be a finite extension such that $F\subseteq KL$ for some field $L\subseteq \kbar$ that is separable over $k$. 
		Then $(F\cap \kbar)/k$ is finite and we have $F=K (F\cap \kbar)$.
	\end{lem}
	
	\begin{proof}
		Notice that, as $F/K$ is finite we may suppose, up to replacing $L$ by a subfield, that $L/k$ is finite. Thus, $L/k$ is primitive. Since we have $K \cap L = k$, then $K \otimes_k L$ is a field by Lemma \ref{Lemma_Linearly_Disjoint_Algebraic_Element}.
		
		We now prove the statement by induction on the degree of $F/K$. If $[F : K] = 1$, the statement is trivial because $F\cap \kbar= k$. Now suppose that $[F:K] > 1$ and that the statement holds for all field extensions of lesser degree. 
		
		Suppose that $F \otimes_k L$ is a field. Then the inclusion $K \to F$ yields an injective ring morphism $\phi_1 : K \otimes_k L \to F \otimes_k L$. We also have injective ring morphisms $\phi_2 : F \otimes_k L \to \overline{K}$ and $\phi: K \otimes_k L \to \overline{K}$, that are given on elementary tensors by $x \otimes y \mapsto xy$. One sees the three morphisms satisfy the compatibility $\phi = \phi_2 \circ \phi_1$. Since $F \subseteq KL$, then $\mathrm{im}(\phi) = KL = \mathrm{im}(\phi_2)$ and hence $\phi_1$ is an isomorphism. Since $L$ is a $k$-module, the functor $\bullet \otimes_k L$ is faithfully flat hence $F = K$, contradicting that $[F:K] > 1$.
		
		If $F \otimes_k L$ is not a field, then, by Lemma \ref{Lemma_Linearly_Disjoint_Algebraic_Element} the field $k$ is not algebraically closed in $F$ and so there is a non-trivial extension $k_1 := F \cap \kbar$ of $k$. So we have $[F:Kk_1] < [F:K]$ and $F \subseteq (K k_1)(Lk_1)$ and $Lk_1/k_1$ is separable. By induction hypothesis the field extension $(F\cap \kbar)/k_1$ is finite and we have $F = (Kk_1)(F\cap \kbar)=K(F\cap \kbar)$.
	\end{proof}
	
	
	\subsection{Algebraic extensions of function fields}
	
	Given two algebraic function fields $L/k_L$ and $K/k$, we say that the former is an \emph{extension} of the latter if $L/K$ and $k_L/k$ are extensions of fields. Moreover, if both $L/K$ and $k_L/k$ are algebraic, we say the extension of function fields is \emph{algebraic}. From now on, an \emph{extension of function fields} will mean an algebraic extension of algebraic function fields. With the above notation we say that $L'/k_{L'}$ is a \emph{subextension} of $L/k_L$ if $L'/k_{L'}$ is an extension of $K/k$ and $L/k_{L}$ is an extension of $L'/k_{L'}$. Suppose that $L/k_L$ is an extension of $K/k$. Then the transcendence degree of $L$ over $k_L$ equals the one of $K$ over $k$. The constant field of $L$ is $\tk_L=\bar{k}\cap L$ and it is algebraic over $k$ (for convenience, the field $k_L$ will not be specified if it equals $k$). In particular, $K \cap \tk_L = \tk$. We call the extension of function fields  \emph{finite} in case $L/K$ and $k_L/k$ are finite. We call the extension of function fields \emph{Galois} if $L/K$ is Galois. In this case, $\tk_L/\tk$ is Galois (see Lemma \ref{Lemma_Separability_from_Function_Fields_Extensions_to_extensions_of_fields_of_constants}) however $k_L/k$ is not necessarily Galois. We call the extension of function fields \emph{Abelian} (respectively, \emph{of exponent} $N$) if it is Galois and the Galois group of $L/K$ is Abelian (respectively, of exponent $N$). Note that for an Abelian extension we have that $\tk_L/k$ is also Abelian while for an Abelian extension of exponent $N$, we have that $\tk_L/k$ is Abelian with exponent dividing $N$.
	
	\begin{rem}
		If $L/k_L$ is an extension of $K/k$, then $L/K$ is finite if and only if $\tk_L/\tk$ is finite if and only if $k_L/k$ is finite. The second implication holds because the extensions $\tk_L/k_L$ and $\tk/k$ are both finite. For the first implication choose transcendence variables $x_1,\ldots, x_n$ for $K/k$, recall that the two extensions $K/ \tk (x_1,\ldots, x_n)$ and $L/ \tk_L (x_1,\ldots, x_n)$ are finite and write
		$$[L:K][K: \tk (x_1,\ldots, x_n)]=[L: \tk (x_1,\ldots, x_n)]=[L: \tk_L (x_1,\ldots, x_n)] [\tk_L:\tk]\,.$$
	\end{rem}
	
	\begin{lem}\label{Lemma_Separability_from_Function_Fields_Extensions_to_extensions_of_fields_of_constants}
		Let $L/k_L$ be an extension of $K/k$.  If $L/K$ is separable, then $\tk_L/\tk$ is separable.
		If $L/K$ is normal, then $\tk_L/k$ and $\tk_L/\tk$ are normal.
		Consequently,
		\begin{enumerate}
			\item[(i)] if $L/K$ is Galois, then $\tk_L/\tk$ is Galois;
			\item[(ii)] if $L/K$ is Galois and $\tk_L/k$ is separable, then $\tk_L/k$ is Galois;
			\item[(iii)] if $L/K$ is separable, then $\tk_L^N/k$ is Galois, where $\tk_L^N$ is the normal closure of $\tk_L$ in $\kbar$.
		\end{enumerate}
	\end{lem}
	\begin{proof}
		We first prove that $\tk_L/\tk$ is separable, assuming that $L/K$ is separable.
		We may suppose that $\tk_L / \tk$ is finite because if this extension is not separable,  then there is some element $\alpha\in \tk_L$ such that $\tk(\alpha)/\tk$ is finite and not separable. 
		We may also suppose that $\tk_L / \tk$ is normal, up to replacing $\tk_L$ by its normal closure in $\bar{k}$.
		Now assume that $\tk_L/\tk$ is not separable. Up to replacing $\tk$ by a separable extension, we may suppose that $\tk_L/\tk$ is purely inseparable. So let $p$ be the finite characteristic of $\tk$ and let $\alpha \in \tk_L \smallsetminus \tk$ be such that $\alpha^{p^m} \in \tk$ holds for some positive integer $m$. As $K \cap \tk_L = \tk$, we have $\alpha \in L \smallsetminus K$ such that $\alpha^{p^m} \in K$, contradicting the separability of $L/K$.
		
		We now assume that $L/K$ is normal and prove that $\tk_L/k$ is normal, which immediately implies that $\tk_L/\tk$ is normal. Consider an element of $\tk_L$, namely some $\alpha\in L$ that is algebraic over $k$. The conjugates of $\alpha$ are in $L$ because $L/K$ is normal and they are all algebraic over $k$, so they are in $\tk_L$.
		The last assertion is immediate.
	\end{proof}
	
	
	\subsection{Constant and  geometric extensions}
	
	We say that an algebraic extension $L/k_L$ of $K/k$ is \emph{constant} if $L = Kk'$ for some algebraic extension $k'$ of $k$. Note that we may replace $k'$ by $k'\tk$ because $L=K(k'\tk)$ and $k'\tk/k$ is algebraic. Hence, we may suppose without loss of generality\ that $\tk \subseteq k'$.
	
	\begin{thm}\label{Theorem_Separable}
		Let $K/k$ be a function field, and let $k'$ be an algebraic extension of $\tk$. If $K$ or $k'$ is separable over $\tk$ (in particular, if $\tk$ is perfect), then the constant field of $K k'$ is $k'$.
	\end{thm}
	\begin{proof}
		Set $L:=Kk'$ and denote by $\tk_L$ its constant field. By Theorem 
		\ref{Theorem_Existence-main} the extension $\tk_L/k'$ is finite and purely inseparable. So it suffices to prove that $\tk_L/k'$ is separable.
		
		Suppose that $k'/\tk$ is separable and fix $\alpha \in k'$ such that $k' = \tk(\alpha)$. Since the minimal polynomial of $\alpha$ over $K$ divides the one over $\tk$, we deduce that $L=K(\alpha)$ is separable over $K$. To conclude remark that, for an element of $\tk_L$, the minimal polynomial over $\tk$ is the same as the one over $K$ by Lemma \ref{Lemma_Linearly_Disjoint_Algebraic_Element}.
		
		Now suppose that $K/\tk$ is separable and let $x_1,\cdots,x_n \in K$ be a transcendental basis over $\tk$. Notice that $K/\tk(x_1,\cdots,x_n)$ is a finite separable extension and hence $L$ is a finite separable extension of $k'(x_1,\cdots,x_n)$. We deduce that $\tk_L(x_1,\cdots,x_n)/k'(x_1,\cdots,x_n)$ is a finite separable extension. So $\tk_L/\tk$ and hence $\tk_L/k'$ is separable by Lemma \ref{Lemma_Separability_from_Function_Fields_Extensions_to_extensions_of_fields_of_constants}.
	\end{proof}
	
	\begin{lem}\label{42}
		If $C/K$ is a finite constant Galois extension and $K'/K$ is a finite Galois extension whose maximal constant subextension is $L/K$, then $CL/K$ is the maximal constant subextension of $CK'/K$.
	\end{lem}
	\begin{proof}
		Write $C=k_CK$ and $L=k_LK$ for some fields $k_C,k_L$ containing $\tk$ and contained in $\bar k$. The extension $CL/K$ is constant because we have $CL=k_Ck_LK$. If it is not maximal, let $M/K$ be the maximal constant subextension of $CK'/K$. By Galois theory, we have $M=CL'$ for some subextension $L'/K$ of $K'/K$. Now, $L'$ strictly contains $L$, and $L'/K$ must be constant by Lemma \ref{algebraic}, contradicting the maximality of $L$.
	\end{proof}
	If $L/k_L$ is an extension of $K/K$, then $\tk \subseteq L \cap \kbar$. This extension is said to be \emph{geometric} if the converse inclusion holds. 
	We call an extension \emph{geometric} if $L\cap \kbar = \tk$ (remark that the inclusion $\supseteq$ always holds). Notice that an algebraic extension is geometric if and only if $\tk_L=\tk$.
	
	\begin{rem}\label{constant-and-geometric}
		An algebraic extension $L/\tk_L$ of $K/\tk$ that is constant and geometric must be trivial. For a constant extension we have $L=Kk'$ for some $\tk \subseteq k'\subseteq \kbar$ and $k'\subseteq (L\cap \kbar) \subseteq  \tk_L$ so, assuming $\tk_L=\tk$, we get $k'=\tk$ and hence $L=K$.
	\end{rem}
	
	\begin{rem}\label{Remark_Degree_Dividing_ell_Constant}
		An algebraic separable extension of function fields of prime degree is constant if and only if it is not geometric. The ``only if'' direction is a consequence of Remark \ref{constant-and-geometric}. Conversely, suppose that the extension is not geometric. So there exists some $\gamma\in \tk_L\setminus \tk$. Since $\tk_L\subseteq \kbar$ we deduce that $\gamma\notin K$ hence it is a primitive element for $L/K$, showing that $L=K \tk(\gamma)$ is constant.
	\end{rem}
	
	Recall that any group whose order is the power of a prime $\ell$ has a normal subgroup of index $\ell$.
	
	\begin{lem}\label{lem_geometric}
		Let $\ell$ be a prime number different from $\ch(k)$. Let $L/k_L$ be a finite Galois extension of $K/k$ such that the degree of $L/K$ is a power of $\ell$. In particular, there exists a Galois subextension of $L/K$ of degree $\ell$. Write $L'/K$ for the largest subextension of exponent $\ell$ (namely, the compositum of all such extensions). 
		Then $L/k_L$  is geometric if and only if $L'/k_{L'}$ is geometric.
	\end{lem}
	\begin{proof}
		If $\tk_L =\tk$, then $\tk_{L'} =\tk$. Conversely, suppose that $\tk_L/\tk$ is not trivial. This extension is finite (because $L/K$ is finite) and Galois (by  Lemma \ref{Lemma_Separability_from_Function_Fields_Extensions_to_extensions_of_fields_of_constants}) and its degree is a power of $\ell$ (by Lemma \ref{Lemma_Linearly_Disjoint_Algebraic_Element}). Then it has a Galois subextension of degree $\ell$. This extension is of the form $\tk(\alpha)/\tk$ where $\alpha\in \kbar \setminus \tk$. Since $\alpha\in L'$ we deduce that $\tk_{L'} \neq  \tk$.
	\end{proof}
	
	
	\subsection{Cyclotomic extensions}
	
	An extension $L/k_L$ of $K/k$ is called \emph{cyclotomic} if $L\subseteq K(\zeta_m)$ holds for some integer $m \geq 1$. Since $K(\zeta_m)/K$ is finite and Abelian, a cyclotomic extension of function fields is finite and Abelian.
	
	In the following results, $\ell$ is a prime number different from $\ch(k)$ and $k(\zeta_{\ell^\infty})=\cup_{n\geq 1} k(\zeta_{\ell^n})$.
	
	\begin{prop}\label{Proposition_Constant_Field_Root_of_Unity}
		Let $K/k$ be a function field. For every $M\geq 1$ that is not divisible by $\ch(k)$, the constant field of $K(\zeta_{M})/k$ is $\tk(\zeta_{M})$. Moreover, the constant field of $K(\zeta_{\ell^\infty})/k$ is $\tk(\zeta_{\ell^\infty})$.
	\end{prop}
	\begin{proof}
		By the assumption on $M$, the field extension $\tk(\zeta_{M})/\tk$ is separable. Hence, by Theorem \ref{Theorem_Separable} the constant field of $K(\zeta_{M})$ is $\tk(\zeta_{M})$. The field $\tk(\zeta_{\ell^\infty})$ is contained in the constant field of $K(\zeta_{\ell^\infty})$ because its elements are algebraic over $k$.
		The converse implication holds because an element of $K(\zeta_{\ell^\infty})$ that is algebraic over $k$ is an element of $K(\zeta_{\ell^m})$ for some integer $m \geq 1$ and hence is contained in the constant field of $K(\zeta_{\ell^m})$, which we have shown to be $\tk(\zeta_{\ell^m})$.
	\end{proof}
	
	\begin{rem}\label{new-rem}
		Let $K/k$ be a field extension such that $k$ is algebraically closed in $K$. Fix some algebraic closure $\overline{K}$ of $K$ and choose an algebraic closure $\kbar$ of $k$ contained in $\overline{K}$.
		Let $F/K$ be a finite extension of degree coprime to $\ch(k)$ that is contained in $K\kbar$. Then Lemma \ref{algebraic} applies. Indeed, it is clear that $F\subset KL'$ for some finite and normal extension $L'$ of $k$. Let $L/k$ be the maximal separable subextension of $L'/k$ hence $L'/L$ is purely inseparable. So
		$KL'/KL$ has degree a power of $\ch(k)$ and it has $FL/KL$ as a subextension. This is only possible if $FL=KL$, giving $F\subseteq KL$.
	\end{rem}
	
	\begin{lem}\label{geom-equiv}
		Let $f \in K^\times$. The following statements are equivalent: 
		\begin{enumerate}
			\item[(i)] There exists $g \in K\kbar$ such that $f = g^\ell$.
			\item[(ii)] For some choice of $\ell$-th root $\sqrt[\ell]{f}$ of $f$, the extension $K(\sqrt[\ell]{f})/K$ is constant.
			\item[(iii)] For every $\ell$-th root $\sqrt[\ell]{f}$ of $f$, the extension $K(\sqrt[\ell]{f})/K$ is constant.
		\end{enumerate}
		Moreover, these statements are implied by 
		\begin{enumerate}
			\item[(iv)] There exist $c \in \tk$ and $b \in K$ such that $f = c b^\ell$.
		\end{enumerate}
		and, if $\zeta_\ell\in K$ (or if $f\in K^{\times \ell}$), they are equivalent to it.
	\end{lem}
	\begin{proof}
		First notice that if $f\in K^{\times \ell}$ then all conditions $(i)$ to $(iv)$ hold (remark that $K(\zeta_\ell)/K$ is constant). So now we may suppose that $f\notin K^{\times \ell}$ hence the polynomial $x^\ell-f$ is irreducible in $K[x]$ by \cite[Theorem 9.1 of Chapter VI]{lang_2002}. In particular, the extension $K(\sqrt[\ell]{f})/K$ has degree $\ell$ for any choice of the $\ell$-th root of $f$.
		
		To complete the proof of the last assertion we show that $(iii)$ implies that there is some $c \in \tk(\zeta_\ell)$ and $b \in K(\zeta_\ell)$ such that $f = c b^\ell$. We may clearly reduce to the case $f\notin K(\zeta_\ell)^{\times \ell}$. Recall from Proposition \ref{Proposition_Constant_Field_Root_of_Unity} that the constant field of $K(\zeta_\ell)$ is $\tk(\zeta_\ell)$.
		By assumption the extension $K(\zeta_\ell, \sqrt[\ell]{f})/K(\zeta_\ell)$
		is constant, so there is an extension $k'/\tk(\zeta_\ell)$ that has degree $\ell$ (by Lemma \ref{Lemma_Linearly_Disjoint_Algebraic_Element}) 
		such that $K(\zeta_\ell, \sqrt[\ell]{f}) = K(\zeta_\ell)k'$. By Kummer theory over $\tk(\zeta_\ell)$ and over $K(\zeta_\ell)$ there is $c\in \tk(\zeta_\ell)^\times$ such that $k'(\zeta_\ell)=\tk(\zeta_\ell, \sqrt[\ell]{c})$ and we have $f c^{-1}\in K(\zeta_\ell)^{\times \ell}$. Thus we
		have $f=c b^\ell$ where $b$ is an $\ell$-th root of $f c^{-1}$ in $K(\zeta_\ell)$.
		
		We are left to prove the equivalence of $(i)$,$(ii)$, and $(iii)$ in the case where $f\notin K^{\times \ell}$, noticing that $(iii) \Rightarrow (ii)$ is obvious. 
		
		$(i) \Rightarrow (ii)$: By Lemma  \ref{algebraic} (that can be applied thanks to Remark \ref{new-rem}) we deduce that $K(g)/K$ is constant, so we conclude by choosing $\sqrt[\ell]{f}=g$.
		
		$(ii) \Rightarrow (i)$: By assumption, $K(\sqrt[\ell]{f}) = K k'$ with $k'$ a subfield of $\kbar$ and we may take $g = \sqrt[\ell]{f}$.
		
		$(ii) \Rightarrow (iii)$: Let $\alpha,\beta \in \overline{K}\setminus K$ be such that $\alpha^\ell = \beta^\ell= f$.
		Suppose that $K(\alpha)/K$ is constant hence $K(\beta)/K$ is contained in the constant extension $K(\alpha, \zeta_\ell)/K$.
		By Lemma  \ref{algebraic} (that can be applied thanks to Remark \ref{new-rem}) we deduce that $K(\beta)/K$ is constant.
	\end{proof}
	
	
	\section{An auxiliary result on function fields extensions}\label{sec:ex-app}
	
	The following result generalizes  \cite[p.254, Theorem 8.4.2]{VillaSalvador_2006}
	from transcendence degree $1$ to any finite transcendence degree. We let $K/k$ be a function field, and we suppose that $k$ is algebraically closed in $K$.
	
	\begin{thm}\label{Theorem_Existence-main} 
		If $k'$ is an extension of $k$, then there exists a function field $L/k_L$ that is an extension of $K/k$ such that $L = k_LK $ and there is a $\tk$-isomorphism $\lambda: k_L \to k'$.
		The above properties imply that $\tk_L$ is a purely inseparable finite extension of $k_L$.
		Finally, if $\hat{L}$, $k_{\hat{L}}$, $\hat{\lambda}$ are also as above, there exists a $K$-isomorphism $\hat{L} \to L$ extending
		$\lambda^{-1} \circ \hat{\lambda}$.
	\end{thm}
	
	\begin{proof}
		Let $n=\trdeg\left(K/k\right)$ and fix a transcendence basis $x_1,\cdots, x_n$ of $K$ over $k$.
		
		Let $\{t_i\}_{i \in I}$ be a transcendence basis of $k'$ over $k$ and choose  elements $\{T_i\}_{i \in I}$ that are algebraically independent over $K$. Mapping $t_i\mapsto T_i$ gives a $k$-morphism from $k'$ to some algebraic closure of $K(\{ T_i \}_{i \in I})$. We denote its image by $k_L$ and we define $L:=k_LK$.

		Let $t \in K\setminus k$ be transcendental over $k$, and suppose that it is algebraic over $k_L$. Then it is algebraic also over $k(\{T_i\}_{i \in I})$ because the extension $k_L/k(\{T_i\}_{i \in I})$ is algebraic. Therefore, we have without loss of generality.\ (and up to clearing the denominators) a relation of the form
		$$
		f_r t^r + \cdots + f_0 = 0
		$$
		with $f_i \in k[T_1,\cdots,T_m]$, not all zero. 
		Some of the $f_i$ must be nonconstant, else $t$ would be algebraic over $k$.  Hence, we have a non-trivial algebraic relation over $K$ of the elements $T_1, \cdots, T_m$, contradicting the algebraic independence of the $T_i$'s. Therefore, we have $k_L \cap K = k$ 
		and $x_1,\ldots, x_n$ are algebraically independent over $k_L$. 
		Now, $k_L$ contains $k$ and we have 
		$$
		[L : k_L(x_1,\ldots, x_n)] \leq [K : k(x_1,\cdots,x_n)] < \infty.
		$$
		Hence, $L/k_L$ is a function field.
		
		We now prove the last assertion. Setting $\theta := \lambda^{-1} \circ \hat{\lambda}$, we prove that we obtain a $K$-isomorphism $\rho: \hat{L} \to L$ by defining
		$$        \rho\left( \frac{\sum_{i=1}^n \alpha_i \beta_i}{\sum_{j=1}^m \gamma_j \delta_j} \right) = \frac{\sum_{i=1}^n \alpha_i \theta(\beta_i)}{\sum_{j=1}^m \gamma_j \theta(\delta_j)}\qquad \text{ where $\alpha_i, \gamma_j \in K$ and $\beta_i, \delta_j \in k_{\hat{L}}$ and $\sum_{j=1}^m \gamma_j \delta_j\neq 0$}.$$
		We claim that for elements $A_i\in K$ and $B_i\in k_{\hat{L}}$ we have 
		\begin{equation}\label{Equation_Well_Defined_and_Isomorphism}
			\sum_{i=1}^n A_i B_i = 0 \text{ if and only if } \sum_{i=1}^n A_i \theta(B_i) = 0\,.
		\end{equation}
		Then the formula for $\rho$ produces an element in $L$ and different expressions for an element of $\hat{L}$ lead to the same image and the map $\rho$ is injective. We immediately see that $\rho$ is a $K$-morphism, and it is surjective because $\theta$ is invertible.
		
		Notice that \eqref{Equation_Well_Defined_and_Isomorphism} only involves a finite number of elements and so we may assume that $k_{\hat{L}}$ and $k_L$ are finitely generated over $k$.
		
		Let $k'/k$ have transcendence basis $t_1,\cdots,t_m$ and that is generated over $k(t_1,\cdots,t_m)$ by some elements $\alpha_i$. Then by construction the elements $T_i=\lambda^{-1}(t_i)$ are a transcendence basis of $k_L$ (and they are algebraically independent over $K$) which generate $k_L$ together with the elements $\lambda^{-1}(\alpha_i)$. The correspondence $t_i\mapsto T_i$ maps the minimal polynomial of $\alpha$ to the one of $\lambda^{-1}(\alpha_i)$.
		The analogous statement holds for $k_{\hat{L}}$, and we deduce that \eqref{Equation_Well_Defined_and_Isomorphism} holds. \bigskip
		
		If $k_L/k$ is purely transcendental with transcendental basis $\{T_i\}_{i\in I}$, by Lemma \ref{Lemma_Algebraically_Closed_Preserved_After_Adding_Variables} the field $k_L = k(T_i)_{i\in I}$ is algebraically closed in $L=K(T_i)_{i\in I}$. So the constant field of $L$ is $k_L$.
		In general, $k_L/k$ will have an intermediate extension $k_L'/k$ that is purely transcendental. By applying what is above, we may replace $K/k$ by $k_L'K/k_L'$ and are left to deal with the case where $k_L/k$ is algebraic. Since $k_LK/k_L$ is a function field, the extension $\tk_L/k_L$ is finite. We are left to prove that $\tk_L/k_L$ is purely inseparable.
		
		So write $\tk_L = k(\gamma_1,\cdots,\gamma_r)$ for some $\gamma_1, \cdots, \gamma_r$ algebraic over $k$. We prove the result by induction on $r$. The result holds trivially for $r = 0$. Now, suppose the result holds for some $r - 1 \geq 0$ and let $k_1 = k(\gamma_1, \cdots, \gamma_{r-1})$ and $L_1 := K k_1 = K(\gamma_1, \cdots, \gamma_{r-1})$. Let $k_2$ be the algebraic closure of $k_1$ in $L_1$. 
		To conclude, it suffices to prove that $\tk_L$ is purely inseparable over $k_2$.
		Since $\gamma_r$ is algebraic over $k$, the same is true for the coefficients of the minimal polynomial $\min(\gamma_r, L_1)$ of $\gamma_r$ over $L_1$. So $\min(\gamma_r,L_1)\in k_2[X]$ and hence $\min(\gamma_r,L_1)=\min(\gamma_r,k_2)$, so we have
		$$
		[L_1(\gamma_r):k_2] = [L_1(\gamma_r):L_1][L_1:k_2] = [k_2(\gamma_r):k_2][L_1:k_2]$$
		and we deduce $[L_1:k_2] = [L_1(\gamma_r):k_2(\gamma_r)]$.
		
		If $c \in \tk_L$ and $p:=\ch(k)$, there is an integer $u \geq 0$ (taking $p^u=1$ if $p=0$) for which the element $c^{p^u}$ is separable over $k_2(\gamma_r)$. By separability, we have  $k_2(\gamma_r,c^{p^u})=k_2(\delta)$ for some primitive element $\delta$. We can write
		$$
		[L_1(\gamma_r):k_2(\gamma_r,c^{p^u})] = [L_1(\delta):k_2(\delta)] = [L_1 : k_2]=[L_1(\gamma_r) : k_2(\gamma_r)],
		$$
		where the second inequality is because a separable extension is linearly disjoint from an inseparable extension. Hence, $c^{p^u} \in k_2(\gamma_r)$, so $c$ is purely inseparable over $k_L=k_2(\gamma_r)$.
	\end{proof}
	
	
	\section{Divisibility modulo constants}\label{sec:div}
	
	
	\subsection{Divisibility parameter modulo constants}
	
	Let $K/k$ be a function field and let $\ell$ be a prime number different from $\ch(k)$. We will see in Lemma \ref{up-to-constants} that for any $\alpha\in K^\times\setminus \tk^\times$ there exists a maximal integer $D\geq 0$ such that $\alpha\in \tk^\times K^{\times \ell^D}$, which we call \emph{$\ell$-divisibility parameter modulo  constants over $K$}. We say that $\alpha \in K^\times$ is \emph{$\ell$-indivisible modulo constants} if $D=0$, which means  $\alpha\notin \tk^\times K^{\times \ell}$.
	
	\begin{lem}\label{up-to-constants}
		Let $\alpha\in K^\times \setminus \tk^\times$. There is an integer $D \geq 0$ which is the largest integer $n$ for which $\alpha\in \tk^\times K^{\times \ell^n}$.
		Moreover, if $\alpha=c \beta^{\ell^n}$ for some $c\in \tk^\times$, for some $\beta\in K^\times$, and for some integer $n \geq 0$, we have $D=n$ if and only if $\beta$ is $\ell$-indivisible modulo constants.
		Consequently, for all but finitely many primes $\ell$, $\alpha$ is $\ell$-indivisible modulo constants.
	\end{lem}
	\begin{proof}
		There is $v \in \Prime(K/k)$ such that $v(\alpha) \neq 0$ (see Lemma \ref{Lemma_Prime_Divisors} (ii)). 
		This shows that $D$ is well-defined and it also proves the last assertion. If $\beta$ is not $\ell$-indivisible modulo constants, we can write $\beta=c_1\beta_1^\ell$ for some $c_1\in \tk^\times$ and for some $\beta_1\in K^\times$. Then $\alpha=(cc_1)\beta_1^{\ell^{n+1}}\in \tk^\times K^{\times \ell^{n+1}}$, implying that $n<D$. Conversely, suppose that $n<D$ and write $\alpha=c_2\beta_2^{\ell^D}$ for some $c_2\in \tk^\times$ and for some $\beta_2\in K^\times$. We can write $\beta^{\ell^n}=(c_2c_1^{-1})\beta_2^{\ell^D}$ hence there is $\gamma\in \kbar$ such that $\beta = \gamma \beta_2^{\ell^{D-n}}$. Since $\gamma \in K \cap \overline{k} = \tk$ we deduce that $\beta$ is not $\ell$-indivisible modulo constants.
	\end{proof}
	
	\begin{lem}\label{same-D}
		Let $\alpha \in K^\times$, and suppose that $\zeta_\ell\in K$. For any $M\geq 1$ not divisible by $\ch(k)$ the $\ell$-divisibility parameter $D$ of $\alpha$ over $K/\tk$ is the same over $K(\zeta_{M})/k(\zeta_{M})$.
		If $L/k_L$ is a finite constant extension of $K$, then the $\ell$-divisibility parameter $D$ of $\alpha$ over $K$ is the same over $L$.
	\end{lem}
	\begin{proof}
		Write $\alpha = c \beta^{\ell^{D}}$ with 
		$c \in \tk^\times$, $\beta\in K^\times$. Since $\beta$ is $\ell$-indivisible modulo constants (and a constant extension of prime degree is geometric), the extension $K(\sqrt[\ell]{\beta})$ is not contained in $K\kbar$.
		We deduce that the extension $L(\zeta_{\ell^{D+1}}, \sqrt[\ell^{D+1}]{\alpha})/K$ is not contained in $K\kbar$, which makes it impossible to have $\alpha = c_L g^{\ell^{D+1}}$ with
		$c_L \in \tk_L^\times$, $g \in L^\times$.
	\end{proof}
	
	\begin{exa}\label{exa-kt}
		Let $K=k(t)$, let $s$ be a positive integer and let $\alpha=c \prod_{i=1}^s P_i^{n_i}$ where $c\in k^\times$, $P_i$ are distinct monic irreducible polynomials in $k[t]$, and $n_i\in \mathbb Z\setminus \{0\}$. Then the  divisibility parameter of $\alpha$ over $K$ is $D_K = \min\{v_\ell(n_i), 1 \leq i \leq s \}$ because $k[t]$ is a unique factorization domain.
		
		Suppose that $c=1$ and let $n$ be a non-zero integer pairwise coprime with each $n_i$.
		Consider the function field $L/k$ such that $L=K(y)$ for some choice of $y$ in an algebraic closure $\bar{K}$ of $K$ such that $y^n = \alpha$ (the function field $L/k$ is called \textit{superelliptic} if $n$ and all $n_i$ are positive).
		
		If $\ell $ divides $n$, then the $\ell$-divisibility parameters of $\alpha$ over $L$ is $\ord_\ell(n)$. If $\ell$ divides the greatest common divisor of the $n_i$'s, then the $\ell$-divisibility parameter of $\alpha$ over $L$ (equivalently, over $K$) is $\min\{v_\ell(n_i), 1 \leq i \leq s \}$.
	\end{exa}
	
	\begin{rem}   
		Suppose that $\alpha\in K^\times\setminus \tk$, and let $n \geq 1$.
		If $L/K$ is a constant extension containing $\zeta_{\ell^n}$, then we have 
		$$[L(\sqrt[\ell^n]{\alpha}):L]=[L(\sqrt[\ell^D]{\alpha}):L]\cdot \ell^{n-D} \qquad \text{for } n>D\,.$$
		Indeed, by Lemma \ref{same-D} the $\ell$-divisibility parameter $D$ over $L$ is the same as over $K$ and hence
		$\alpha\notin L^{\times \ell^{D+1}}$. We conclude by Kummer theory.
	\end{rem}
	
	
	\subsection{Independence modulo constants}
	
	Let $K/k$ be a function field and let $n \geq 1$ be an integer coprime with $\ch(K)$. If $G$ is a subgroup of $K^\times$, then we write $K(\sqrt[n]{G})$ for the extension of $K$ generated by all the $\sqrt[n]{\alpha}$, with $\alpha \in G$. This extension is uniquely determined by $G$, as a subfield of a fixed algebraic closure $K$. In particular, if $\{\alpha_1,\cdots,\alpha_r\}$ is a minimal subset of $K^\times$ of generators for $G$, then by \cite[Corollary 6.1.16]{lang_2002} we have
	$$
	K(\sqrt[n]{G}) = K(\sqrt[n]{\alpha_1},\cdots,\sqrt[n]{\alpha_r}) = K(\sqrt[n]{\alpha_1}) \cdot \cdots \cdot K(\sqrt[n]{\alpha_r})
	$$
	and we have a canonical group isomorphism
	$$
	\mathrm{Gal}(K(\sqrt[n]{G})/K) \xrightarrow{\simeq} \mathrm{Gal}(K(\sqrt[n]{\alpha_1})/K) \times \cdots \times \mathrm{Gal}(K\sqrt[n]{\alpha_r})/K).
	$$
	We now suppose that $n$ is the power of some prime number $\ell$ and hence we will suppose
	that $\zeta_\ell\in K$ (as we are interested in the Kummer extensions, we may replace $K$ by $K(\zeta_\ell)$).
	
	\begin{defi}
		We say elements $\alpha_1,\ldots, \alpha_r\in K^\times$ are \emph{$\ell$-independent modulo constants} if for each vector $(x_1,\cdots,x_r) \in \Z^r$ not in $(\ell \Z)^r$, the element $\prod_{i=1}^r \alpha_i^{x_i}$ is $\ell$-indivisible modulo constants.
	\end{defi}
	
	\begin{lem}
		Let $\alpha_1,\cdots,\alpha_r \in K^\times$ be $\ell$-independent modulo constants. The group generated by the images of $\alpha_1,\cdots,\alpha_r$ in the quotient group $\frac{K^\times}{\tk^\times}$ has rank $r$. In particular, the $\alpha_i$'s generate a subgroup $G$ of $K^\times$ that is torsion-free and of rank $r$. 
	\end{lem}
	\begin{proof}
		The latter assertion is readily seen to be a consequence of the former.
		Suppose that $\eta := \prod_{i=1}^r \alpha_i^{x_i}$ for some non-zero vector $(x_1,\cdots,x_r) \in \Z^r$. Up to replacing $\eta$, we may suppose without loss of generality that not all $x_i$ are divisible by $\ell$. Since $\eta$ is not $\ell$-indivisible modulo constants, this contradicts the assumption of $\ell$-independence modulo constants.
	\end{proof}
	
	The function fields appearing in the next theorem are implicitly considered as function fields over the same field $k$.
	
	\begin{thm}\label{Theorem_Equivalent_ell_Independent}
		Let $G$ be a torsion-free subgroup of $K^\times$ of rank $r > 0$, and write $G=\langle \alpha_1, \ldots, \alpha_r\rangle$ with $\alpha_i\in K^\times$. The following are equivalent:
		\begin{itemize}
			\item[(i)] The elements $\alpha_1,\ldots, \alpha_r$ are $\ell$-independent modulo constants.
			\item[(ii)] The Kummer extension $K(\sqrt[\ell]{G})/K$ has degree $\ell^r$ and it is geometric.
			\item[(iii)] For every $m\geq n\geq 1$ the Kummer extension $K(\zeta_{\ell^m}, \sqrt[\ell^n]{G})/K(\zeta_{\ell^m})$ has degree $\ell^{rn}$ and it is geometric.
			\item[(iv)] For every $n\geq 1$ the Kummer extension $K(\zeta_{\ell^\infty}, \sqrt[\ell^n]{G})/K(\zeta_{\ell^\infty})$ has degree $\ell^{rn}$ and it is geometric.
			\item[(v)] For every $n\geq 1$ the Kummer extension $\kbar K(\sqrt[\ell^n]{G})/\kbar K$ has degree $\ell^{rn}$ and it is geometric.
		\end{itemize}
		(As the given Kummer extensions have exponent dividing $\ell^n$ and rank at most $r$, their degree being $\ell^{rn}$ implies that their Galois group is $(\mathbb Z/\ell^{n} \mathbb Z)^r$.)
	\end{thm}
	\begin{proof}
		\emph{(iii) $\Leftrightarrow$ (iv)}  Fix $n\geq 1$. By Kummer theory, the degrees of the given extensions do not exceed $\ell^{rn}$. The degree of $K(\zeta_{\ell^\infty}, \sqrt[\ell^n]{G})/K(\zeta_{\ell^\infty})$ divides  the one of $K(\zeta_{\ell^m}, \sqrt[\ell^n]{G})/K(\zeta_{\ell^m})$. Conversely, since the degree in $(iii)$ is constant in $m$, the two extensions $K(\zeta_{\ell^\infty})$ and $K(\zeta_{\ell^n}, \sqrt[\ell^n]{G})$ are linearly disjoint over $K(\zeta_{\ell^n})$ and hence the degree of the Kummer extension equals  $\ell^{rn}$ also over $K(\zeta_{\ell^\infty})$. Now we may suppose that the given extensions have degree $\ell^{nr}$ and we only need to show that if one of them is geometric, the other is also geometric. By Proposition \ref{Proposition_Constant_Field_Root_of_Unity}, the constant field of $K(\zeta_{\ell^m})/k$ is $\tk(\zeta_{\ell^m})$ while the one of $K(\zeta_{\ell^\infty})/k$ is 
		$\tk(\zeta_{\ell^\infty})$ by Proposition \ref{Proposition_Constant_Field_Root_of_Unity}. Moreover, we can write $K(\zeta_{\ell^\infty},\sqrt[\ell^n]{G})=\bigcup_{m=1}^\infty K(\zeta_{\ell^m}, \sqrt[\ell^n]{G})$. 
		
		Suppose that all extensions $K(\zeta_{\ell^m}, \sqrt[\ell^n]{G})/K(\zeta_{\ell^m})$ are geometric. An element  $\gamma\in K(\zeta_{\ell^\infty}, \sqrt[n]{G})$ that is algebraic over $k$ belongs to some field $K(\zeta_{\ell^m}, \sqrt[\ell^n]{G})$.
		Our assumption on this field then implies $\gamma\in \tk(\zeta_{\ell^m})$. This shows that $K(\zeta_{\ell^\infty}, \sqrt[n]{G})\cap \kbar = \tk(\zeta_{\ell^\infty})$.
		
		Now suppose that $K(\zeta_{\ell^\infty},\sqrt[n]{G})/K(\zeta_{\ell^\infty})$ is geometric. An element of $K(\zeta_{\ell^m},\sqrt[n]{G})$ that is algebraic over $k$ belongs to $\tk(\zeta_{\ell^\infty})$ and hence to $\tk(\zeta_{\ell^u})$ for some $u\geq m$.
		We conclude by proving that $K(\zeta_{\ell^m}, \sqrt[n]{G})\cap \tk(\zeta_{\ell^u})=\tk(\zeta_{\ell^m})$. If we had $\zeta_{\ell^{m+1}}$ in $K(\zeta_{\ell^m}, \sqrt[n]{G})\setminus K(\zeta_{\ell^m})$, then the degree of $K(\zeta_{\ell^m}, \sqrt[n]{G})/K(\zeta_{\ell^m})$ would not be the same as the one of $K(\zeta_{\ell^\infty}, \sqrt[n]{G})/K(\zeta_{\ell^\infty})$, contradiction.
		
		\emph{(iii) $\Rightarrow$ (ii)} This is immediate by taking $m=n=1$.
		
		\emph{(ii) $\Rightarrow$ (iii)} Since $K(\sqrt[\ell]{G})/K$ is geometric, we deduce that $K(\zeta_{\ell^m}, \sqrt[\ell]{G})/K(\zeta_{\ell^m})$ has also degree $\ell^r$. By Kummer theory, we deduce that $K(\zeta_{\ell^m}, \sqrt[\ell^n]{G})/K(\zeta_{\ell^m})$ has degree $\ell^{rn}$.
		By Lemma \ref{lem_geometric} it suffices to prove that $K(\zeta_{\ell^m}, \sqrt[\ell]{G})/K(\zeta_{\ell^m})$ is geometric. This extension is a tower of extensions of degree $\ell$ that are nonconstant and hence geometric by Remark \ref{Remark_Degree_Dividing_ell_Constant}, so we may conclude.
		
		\emph{(i) $\Leftrightarrow$ (ii)} Property $(ii)$ is equivalent to the fact that for every $\alpha\in G$ that is not an $\ell$-th power in $G$ the extension $K(\sqrt[\ell]{\alpha})/K$ is geometric of degree $\ell$ (because the Galois group of $K(\sqrt[\ell]{G})/K$ has exponent dividing $\ell$, and an extension of degree $\ell$ is geometric if and only if it is not constant). Recalling that $\zeta_\ell\in K$, by Lemma \ref{geom-equiv} the above property precisely means that every element $\alpha:=\prod_{i=1}^r \alpha_i^{x_i}$ is $\ell$-indivisible up to constants whenever the integers $x_i$'s are not all divisible by $\ell$.

		\emph{(i) $\Leftrightarrow$ (v)}. The extension $\kbar K(\sqrt[\ell^n]{G})/\kbar K$ is always geometric. Indeed, the constant field of $\kbar K(\sqrt[\ell^n]{G})$ is a finite extension of the constant field of $\kbar K$ hence it is $\kbar$. Property (i) means that any element $\alpha := \prod_{i=1}^r \alpha_i^{x_i}$ (where the integers $x_i$'s are not all divisible by $\ell$) is $\ell$-indivisible modulo constants. Property (v) (the assertion on the degree) means, by Kummer theory, that any such element is not an $\ell$-th power in $\kbar K$. We may conclude by Lemma \ref{geom-equiv} \emph{(i) $\Leftrightarrow$ (iv)}.
	\end{proof}
	
	\begin{cor}\label{cor-Kummerdeg-modulo}
		Suppose that $\alpha_1,\ldots, \alpha_r\in K^\times$ can be expressed as
		$\alpha_i= c_i \beta_i^{\ell^{D_i}}$, where $c_i\in \tk$ and where the elements $\beta_1, \ldots, \beta_r$ are 
		$\ell$-independent modulo constants over $K$. 
		For every $n\geq 1$ and for every $M$ such that $\ell^n \mid M$,  the maximal constant subextension of $K(\zeta_{M}, \sqrt[\ell^n]{G})/K$ is 
		$$L_{M,n}:=K(\zeta_{M}, \sqrt[\ell^{\min(D_1, n)}]{\alpha_1}, \ldots, \sqrt[\ell^{\min(D_r, n})]{\alpha_r})\,.$$ 
		Moreover, we have
		$$\Gal(K(\zeta_{M}, \sqrt[\ell^n]{G})/L_{M,n})\simeq \Gal(\kbar K(\sqrt[\ell^n]{G})/\kbar K)\simeq \prod_{i=1}^r \mathbb Z/\ell^{\max(n-D_i,0)} \mathbb Z\,.$$
	\end{cor}
	\begin{proof}
		By definition of the $\ell$-divisibility parameters, $L_{M,n}$ is a constant extension of $K$ contained in  $K(\zeta_{M}, \sqrt[\ell^n]{G})$. Remark that we have 
		$$\kbar K(\sqrt[\ell^n]{G})=\overline{k}K(\sqrt[\ell^{\max(n-D_1, 0)}]{\beta_1}, \ldots, \sqrt[\ell^{\max(n-D_r, 0)}]{\beta_r})\,.$$
		Theorem \ref{Theorem_Equivalent_ell_Independent} ensures that the latter field is a geometric extension of $\overline{k}K$ whose Galois group is isomorphic to $\prod_{i=1}^r \mathbb Z/\ell^{\max(n-D_i,0)} \mathbb Z$.
		To conclude, it suffices to show that 
		$K(\zeta_{M}, \sqrt[\ell^n]{G})\cap \overline{k}K =L_{M,n}$.
		This is the case because the degree of $K(\zeta_{M}, \sqrt[\ell^n]{G})/L_{M,n}$ is a multiple of $\bar k K(\sqrt[\ell^n]{G})/\bar k K$ hence it equals $\sum_{i=1}^r \max(n-D_i,0)$
		(it is at most this number by Kummer theory as $\sqrt[\ell^n]{\alpha_i}$ is an $\ell^{\max(n-D_i,0)}$-th root of $\sqrt[\ell^{\min(D_i, n)}]{\alpha_i}$).
	\end{proof}
	
	
	\subsection{Good basis modulo constants}
	
	In the following definition, we let $D_i$ be the $\ell$-divisibility parameters of $\alpha_i$ over $K$.
	
	\begin{defi}\label{good-basis-modulo}
		Consider a finitely generated subgroup $G$ of $K^\times$ that intersects $\tk$ only trivially and that has positive rank $r$.
		Let $\{\alpha_1,\ldots, \alpha_r\}$ be a basis of $G$. This basis is said to be an \emph{$\ell$-good basis modulo constants} if, equivalently:
		\begin{itemize}
			\item[(i)] the quantity $\sum_{i=1}^r D_i$ is maximal among the bases of $G$;
			\item[(ii)] we can write $\alpha_i= c_i \beta_i^{\ell^{D_i}}$ where $c_i\in \tk$ and $\beta_i\in K$ and $\beta_1, \ldots, \beta_r$ are $\ell$-independent modulo constants.
		\end{itemize}
		The multiset of $D_1, \cdots, D_r$ is uniquely determined by $K$ and $G$ (see Remark \ref{Remark_Multiset_Good_Basis_Modulo_tk}). We call it the \emph{$\ell$-divisibility parameters modulo constants of $G$ over $K$}.
	\end{defi}
	
	\begin{prop}\label{Proposition_Equivalence_of_Definitions}
		Definition \ref{good-basis-modulo} is well-posed, namely the equivalence holds true.
	\end{prop}
	\begin{proof}
		\emph{(i)$\Rightarrow$(ii)} By the definition of the $\ell$-divisibility parameters we can write $\alpha_i= c_i \beta_i^{\ell^{D_i}}$ where $c_i\in \tk$ and $\beta_i\in K$. Suppose that the $\beta_i$'s are not $\ell$-independent modulo constants. So there are a non-empty subset $J\subseteq \{1,\ldots, r\}$ and integers $x_j$ coprime to $\ell$ such that 
		$\beta:=\prod_{j\in J} \beta_j^{x_j}$ is not $\ell$-indivisible modulo constants. 
		We may suppose without loss of generality that $1\in J$ and $D_1=\max\{ D_j : j \in J \}$. We may additionally suppose that $x_1=1$ (by raising $\beta$ to some power $y_1$ such that $x_1y_1\equiv 1 \bmod \ell$ and then discarding an $\ell$-th power).
		Define $$\alpha:=\prod_{j\in J} \alpha_j^{x_j \ell^{D_1-D_j}}\,.$$ By inspecting the factor  $j=1$, we may replace $\alpha_1$ by $\alpha$ and still get a basis of $G$. Moreover, by construction 
		$$
		\alpha \beta^{-\ell^{D_1}}=\prod_{j\in J} \alpha_j^{x_j \ell^{D_1-D_j}} \left( \prod_{j\in J} \beta_j^{x_j\ell^{D_1}} \right)^{-1} = \prod_{j\in J} c_j^{x_j \ell^{D_1-D_j}}
		$$ is a constant hence (since $\beta$ is not $\ell$-indivisible modulo constants) the first divisibility parameter of $\alpha$ is strictly larger than $D_1$, contradicting the maximality of $\sum_{i=1}^r D_i$.
		
		\emph{(ii)$\Rightarrow$ (i)}
		We cannot have a different basis $\{\alpha_1',\cdots,\alpha'_r\}$ of $G$ such that $\sum_{i=1}^r D'_i>\sum_{i=1}^r D_i$. This is because, if $n$ is sufficiently large,  
		by Corollary \ref{cor-Kummerdeg-modulo} the degree of 
		$\kbar K(\sqrt[\ell^n]{G})/\kbar K$
		has $\ell$-adic valuation $nr-\sum_{i=1}^r D_i$ while $$[\kbar K(\sqrt[\ell^n]{G}):\kbar K]=[\kbar K(\sqrt[\ell^{\max (n-D'_1,0)}]{\alpha'_1}, \ldots, \sqrt[\ell^{\max (n-D'_r,0)}]{\alpha'_r}):\kbar K]\leq nr-\sum_{i=1}^r D'_i\,.$$
	\end{proof}
	
	\begin{thm}\label{existbasis}
		Consider a finitely generated subgroup $G$ of $K^\times$ of rank $r > 0$ that intersects $\tk$ only trivially. Then the following statements hold.
		\begin{itemize}
			\item[(i)] The group $G$ has an $\ell$-good basis modulo constants.
			\item[(ii)] For all sufficiently large $n$ we have 
			\begin{equation}\label{eq-growth}
				[\kbar K (\sqrt[\ell^{n+1}]{G}):\kbar K (\sqrt[\ell^n]{G})]=\ell^{r}\,.
			\end{equation}
			\item[(iii)] There exists an integer $\varepsilon\geq 0$ such that for all $n>\varepsilon$ we have
			$$G\cap \tk^\times K^{\times \ell^n} \subseteq G^{\ell^{n-\epsilon}}\,.$$
		\end{itemize}
	\end{thm}
	\begin{proof}
		Observing that $(\bar k K )^{\times \ell^n}=\bar k^\times (K ^{\times \ell^n})$, we have $G\cap (\bar k K )^{\times \ell^n}=G\cap \tk^\times K^{\times \ell^n}$ hence $(iii)$ is equivalent to $(ii)$ by Kummer theory over $\kbar K$.
		If $G$ has an $\ell$-good basis modulo constants, then  $(ii)$ is proven in Corollary \ref{cor-Kummerdeg-modulo}. The converse implication is because, if there is a basis such that $\sum_i D_i > (n+1)r$, then \eqref{eq-growth} does not hold.
		
		We conclude by proving $(iii)$. By applying Lemma \ref{Lemma_Prime_Divisors} to a set of generators of $G$, there is a finite non-empty subset $S$ of $\mathrm{Prime}(K/k)$ of cardinality $\# S$ consisting of the divisors $v$ such that there is some $g\in G$ such that $v(g)\neq 0$. Consider the group morphism 
		$$\Phi: G\rightarrow \mathbb Z^{\#S}\qquad g\mapsto (v(g))_{v\in S}\,.$$
		The kernel of $\Phi$ is trivial because by Lemma \ref{Lemma_Prime_Divisors} it consists of constants, thus $\Phi(G)$ is a subgroup of $\mathbb Z^{\#S}$ of rank $r$. By choosing an appropriate basis of $G$, the basis of $\Phi(G)$ consists of the column vectors of the following matrix, where the integers $c_1,\ldots, c_r$ are the non-zero entries:
		$$\begin{pmatrix}
			
			c_1 & & & \\
			& c_2& & \\
			& & \ddots & \\
			& & & c_r \\
			&  &  &   \\
			&  &  &   \\
		\end{pmatrix}\,.$$
		Let $\varepsilon$ be the maximum of the $\ell$-adic valuation of the integers $c_i$'s and fix $n>\varepsilon$.
		If $g\in G$ is such that $\ell^n\mid  v(g)$ for every $v\in S$, then all the coordinates of $\Phi(g)$ with respect to the above basis are divisible by $\ell^{n-\varepsilon}$, so we deduce that $g\in G^{\ell^{n-\epsilon}}$.
	\end{proof}
	
	\begin{thm}\label{exists}
		Let $K/k$ be a function field and let $G$ be a finitely generated subgroup of $K^\times$. An $\ell$-good basis modulo constants exists for $G$ if and only if the rank of $G$ is the same as the rank of $G$ modulo constants.
	\end{thm}
	\begin{proof}
		The sufficiency has been shown in Theorem \ref{existbasis} and the necessity follows from Corollary \ref{cor-Kummerdeg-modulo} by working over $\kbar K$.
	\end{proof}
	
	\begin{rem}\label{Remark_Multiset_Good_Basis_Modulo_tk}
		The multiset of parameters $D_i$ for an $\ell$-good basis modulo constants is uniquely determined by $K$ and $G$. This is clear from the Galois group structure in Corollary \ref{cor-Kummerdeg-modulo} by taking $n>\max\{ D_i : 1 \leq i \leq r \}$.
		Provided that $G$ admits an $\ell$-good basis modulo constants, there is an easy algorithm to construct one from any given basis of $G$, which is based on the following observation: if the elements of a given basis of $G$ are not $\ell$-independent modulo constants, then there is a way to modify one element of the basis and increase $\sum_i D_i$ (following the strategy in the proof of \cite[Theorem 14]{DebryPerucca}, which also resembles the proof of Proposition \ref{Proposition_Equivalence_of_Definitions} $(i) \Rightarrow (ii)$).
	\end{rem}
	
	
	\section{Reducing to cyclotomic-Kummer extensions of the constant field}\label{sec:Kum}
	
	Consider a function field $K/k$ and a Kummer extension of the form
	$$K(\zeta_{M}, \sqrt[N]{G})/K(\zeta_{M})\,,$$
	where $G$ is a finitely generated subgroup of $K^\times$ and $M$ and $N$ are positive integers such that $M$ is divisible by $N$ but not by $\ch(K)$. In this section, we precisely show how the questions of computing the degree of this extension and of computing the group structure of its Galois group reduce to the corresponding questions for the Kummer extension of the constant fields.
	
	Considering the prime factorization $N=\prod_{\ell \mid N} \ell^{v_\ell(N)}$, by  Kummer theory (notice the pairwise coprime exponents of the Kummer extensions in the product group) we have a group isomorphism
	$$\Gal(K(\zeta_{M}, \sqrt[N]{G})/K(\zeta_{M}))\simeq \prod_{\ell\mid N} \Gal(K(\zeta_{M}, \sqrt[\ell^{v_\ell(N)}]{G})/K(\zeta_{M}))\,.$$
	We will therefore restrict to the case where $N=\ell^n$
	for some fixed prime number $\ell$ different from $\ch(k)$ and for some integer $n\geq 1$. We may then without loss of generality replace $K$ by $K(\zeta_\ell)$. We may also suppose that $G$ is not a subgroup of $\tk^\times$. Hence, from now on $G_0 := G \cap \tk$ is a proper subgroup of $G$ and we can write the set decomposition $G=G_0G'$, where $G'$ is a subgroup of $G$ isomorphic to $G/G_0$. Let $r>0$ be the rank of $G'$ and let $D_1, \ldots, D_r$ be the $\ell$-divisibility parameters modulo constants of $G'$ over $K$. Up to relabelling, we suppose without loss of generality that $D_1\leq D_2\leq \cdots \leq D_r$.
	Let $\alpha_i=c_i \beta_i^{\ell^{D_i}}$ be the elements of an $\ell$-good basis modulo constants of $G'$ over $K$, with $c_i\in \tk$ and $\beta_i\in K$. 
	
	The following result reduces the problem of determining the degree of the Kummer extension $K(\zeta_{M}, \sqrt[\ell^{n}]{G})/K(\zeta_{M})$ to determine the degree of a Kummer extension that only depends on $\tk(\zeta_{M})$, $\ell^n$, $G_0$, the $c_i$'s and the $D_i$'s:
	
	\begin{thm}\label{Main_Theorem_Part_2}
		Define $H$ as the subgroup of $K^\times$ generated by $G_0$ and by the elements $c_i^{\ell^{\max(n-D_i,0)}}$.
		Then we have 
		$$[K(\zeta_{M}, \sqrt[\ell^{n}]{G}):K(\zeta_{M})]=[\tk(\zeta_{M}, \sqrt[\ell^{n}]{H}):\tk(\zeta_{M})]\cdot \prod_{i=1}^r \ell^{\max(n-D_i,0)}\,.$$
	\end{thm}
	\begin{proof} 
		Remark that, by Lemma \ref{Lemma_Linearly_Disjoint_Algebraic_Element}  applied to the monogenic extension $\tk(\zeta_{M}, \sqrt[\ell^{n}]{H})/\tk(\zeta_{M}$) we have 
		$$[\tk(\zeta_{M}, \sqrt[\ell^{n}]{H}):\tk(\zeta_{M})]=[K(\zeta_{M}, \sqrt[\ell^{n}]{H}):K(\zeta_{M})]\,.$$
		We may conclude by Corollary \ref{cor-Kummerdeg-modulo} (applied to $G'$) because the largest constant subextension of 
		$K(\zeta_{M}, \sqrt[\ell^{n}]{G})/K(\zeta_{M})$
		is generated by the $\ell^n$-th roots of $H$, in view of Lemma \ref{42} (where  $C$ is the Kummer extension generated by $G_0$ and  $K'$ is the Kummer extension generated by $G'$).
	\end{proof}
	
	\begin{defi}\label{di}
		We define elements $d_1,\ldots, d_r$ in $\mathbb Z_{\geq 0}\cup \{+\infty\}$ as follows: $d_i$ is maximal such that the image of $c_i$ in the quotient group $\tk(\zeta_{M})^{\times} / \langle G_0, c_j : j>i\rangle$ is an $\ell^{d_i}$-th power.
	\end{defi}
	
	\begin{rem}
		According to the use that we make of $d_1, \cdots, d_r$, we do not need their precise value if we know that $d_i\geq D_i$ for all $i$. However, we could compute $d_i$ by multiplying $c_i$ by an element of $\langle G_0, c_j : j>i\rangle$ (and using powers of the generators with exponents from $0$ to $\ell^{D_i}-1$). Notice that, if $G_0$, the $D_i$'s and the $c_i$'s are known in advance, then calculating the parameters $d_i$ only involves computations over $\tk(\zeta_M)$.
	\end{rem}
	
	\begin{thm}\label{special}
		With the above notation, we have a group isomorphism 
		$$\Gal\left(K(\zeta_{M}, \sqrt[\ell^{n}]{G})/K(\zeta_{M})\right)\simeq \Gal\left(\tk(\zeta_{M}, \sqrt[\ell^{n}]{G_0})/\tk(\zeta_{M})\right)\times  \prod_{i=1}^r \mathbb Z/ \ell^{\max (n-\min(D_i, d_i),0)} \mathbb Z\,.$$
	\end{thm}
	\begin{proof}
		By Lemma \ref{Lemma_Linearly_Disjoint_Algebraic_Element} applied to the monogenic Kummer extension $\tk(\zeta_M, \sqrt[\ell^n]{G_0})/\tk(\zeta_M)$ we have 
		$$\Gal(\tk(\zeta_{M}, \sqrt[\ell^{n}]{G_0})/\tk(\zeta_{M}))\simeq \Gal(K(\zeta_{M}, \sqrt[\ell^{n}]{G_0})/K(\zeta_{M}))\,.$$
		By Lemma \ref{same-D} the $\ell$-divisibility parameters of $G_0$ over $K$ are the same over $K(\zeta_{M})$ and by Proposition \ref{Proposition_Constant_Field_Root_of_Unity} the constant field of $K(\zeta_{M})/\tk$ is $\tk(\zeta_{M})$. We then work over $K(\zeta_{M})/\tk(\zeta_{M})$.
		
		Notice that multiplying $\alpha_i$ by an element of $G_0$ does not change the $\ell$-divisibility parameter. Also, since the $\beta_i$'s are $\ell$-independent modulo constants and since we assume $D_i\leq D_j$ for $i<j$, then we can replace $\alpha_i$ by $\alpha_i \alpha_j^e$ for any $e\in \mathbb Z$ without altering the $i$-th $\ell$- divisibility parameter. We perform such multiplications, modifying the elements $\alpha_r,\alpha_{r-1}, \ldots, \alpha_1$ in this order. In this way, we can achieve that 
		$c_i$ is an $\ell^{D_i}$-th power in $\tk(\zeta_{M})^{\times}$ or that the largest integer $d_i$ such that $c_i$ is an $\ell^{d_i}$-th power in $\tk(\zeta_{M})^{\times}$ is strictly less than $D_i$ and is as in Definition \ref{di}. 
		Consider the tower of Kummer extensions
		$$
		L(\sqrt[\ell^{n}]{G_0}) \subseteq L(\sqrt[\ell^{n}]{G_0},\sqrt[\ell^{n}]{\alpha_r})\subseteq L(\sqrt[\ell^{n}]{G_0},\sqrt[\ell^{n}]{\alpha_r}, \sqrt[\ell^{n}]{\alpha_{r-1}})\subseteq \cdots \subseteq L(\sqrt[\ell^{n}]{G})\,.
		$$
		By definition of $d_i$ (and by Kummer theory over $L$, recalling that $\sqrt[\ell]{\beta_i}\notin L$) the Galois group of 
		$L(\sqrt[\ell^{n}]{\alpha_i})/L$ is isomorphic to $\mathbb Z/\ell^{\max(n-\min(D_i,d_i),0)}\mathbb Z$.
		To conclude, we only need to prove that this extension has the same degree as 
		$$L(\sqrt[\ell^{n}]{\alpha_i},  \sqrt[\ell^{n}]{G_0},\sqrt[\ell^{n}]{\alpha_j}: j>i)/L(\sqrt[\ell^{n}]{G_0},\sqrt[\ell^{n}]{\alpha_j}: j>i)\,.$$
		
		By the radical correspondence of Kummer theory, this is equivalent to saying that there are no integers $x$ satisfying both $$
		\alpha_i^x \mod L^{\times \ell^n} \in \langle G_0, \alpha_j : j>i \rangle \bmod L^{\times \ell^n}\,\quad \text{and}\quad \alpha_i^x\notin L^{\times \ell^n}\,.
		$$
		If $d_i\geq D_i$ we conclude because, by Theorem \ref{Theorem_Equivalent_ell_Independent} (applied to the $\beta_i$'s),
		no integer $z$ satisfies both 
		\begin{equation}\label{beta}
			\beta_i^z \bmod L^{\times \ell^n} \in \langle \tk^\times, \beta_j : j\neq i \rangle \bmod L^{\times \ell^n}\,\quad \text{and}\quad \beta_i^z\notin L^{\times \ell^n}\,.
		\end{equation}
		
		Now suppose that $d_i<D_i$. We have $\alpha_i^x =c_i^{\ell^y}\beta_i^{\ell^{D_i+y}}$ and  without loss of generality, that $x=\ell^y$ for some integer $y \geq 1$.
		If $D_i+y<n$ we conclude because there is no integer $z$ as in \eqref{beta}. Now suppose that $D_i+y\geq n$
		and hence $\alpha_i^x\equiv c_i^x \bmod L^{\times \ell^n}$.
		We are left to prove that we cannot have
		$$
		c_i^x \bmod L^{\times \ell^n} \in \langle G_0, \alpha_j : j>i \rangle \bmod L^{\times \ell^n}\,\quad \text{and}\quad c_i^x\notin L^{\times \ell^n}\,.
		$$
		Notice that, since $\tk(\zeta_{M})$ is the constant field of $L$, for any $\gamma\in \tk(\zeta_{M})^\times$, we have $\gamma\in L^{\times \ell^n}$
		if and only if 
		$\gamma\in \tk(\zeta_{M})^{\times \ell^n}$. We claim that if
		$c_i^x \bmod L^{\times \ell^n}$ is in $\langle G_0, \alpha_j : j>i \rangle \bmod L^{\times \ell^n}$, then it is in $\langle G_0, c_j : j>i \rangle \bmod L^{\times \ell^n}$.
		This will conclude because no integer $x$ can satisfy both 
		$$
		c_i^x \bmod \tk(\zeta_{M})^{\times \ell^n} \in \langle G_0, c_j : j>i \rangle \bmod \tk(\zeta_{M})^{\times \ell^n}\,\quad \text{and}\quad c_i^x\notin \tk(\zeta_{M})^{\times \ell^n}.
		$$
		Indeed, by  construction, the divisibility of $c_i$ is the same in $\tk(\zeta_{M})^{\times}$ as the one of its image in the quotient group $\tk(\zeta_{M})^{\times} /\langle G_0, c_j : j>i \rangle$.
		The claim holds  
		because Theorem \ref{Theorem_Equivalent_ell_Independent}  (applied to the $\beta_i$'s) gives that an integer $z$ such that $\alpha_j^z \bmod L^{\times \ell^n}$ is in $\langle \tk, \alpha_h : h\neq j \rangle \bmod L^{\times \ell^n}$ must satisfy   $v_\ell(z)+D_i\geq n$, and hence $\alpha_j^z \equiv c_j^z \bmod L^{\times \ell^n}$.
	\end{proof}
	
	\begin{rem}\label{zero}
		The extension $\tk(\zeta_M, \sqrt[\ell^n]{G_0})/\tk(\zeta_M)$ is finite and separable, hence, monogenic. Therefore, by Lemma \ref{Lemma_Linearly_Disjoint_Algebraic_Element} we have
		$$
		[K(\zeta_M, \sqrt[\ell^n]{G_0}) : K(\zeta_M)] = [\tk(\zeta_M, \sqrt[\ell^n]{G_0}) : \tk(\zeta_M)].
		$$
	\end{rem}
	
	\begin{exa}
		We can apply the above results multiple times if $\tk$ is a function field. For example, let $K=\tk(S)$ where $\tk=\mathbb Q(T)$. Then by Theorem \ref{Main_Theorem_Part_2} we may first reduce to $\mathbb Q(T)$ and then to $\mathbb Q$. Alternatively, we can see $K$ as a function field over $\mathbb Q$ and apply Theorem \ref{Main_Theorem_Part_2} to reduce directly from $K$ to $\mathbb Q$. This is possible because the Kummer extensions $ K(\zeta_{M}, \sqrt[N]{G})/K(\zeta_{M})$ do not depend on the choice of the constant field.
	\end{exa}
	
	\begin{exa}
		We use the notation of Example \ref{exa-kt}. Fix two distinct monic irreducible polynomials $P_1$ and $P_2$ and let $c_1,c_2\in k^\times$. We consider the group 
		$$G=\langle c_1 P_1 , c_2 P_1P_2^2\rangle \,.$$
		Both generators are $2$-indivisible modulo constants. We also have $G=\langle \alpha_1, \alpha_2 \rangle$, where 
		$$\alpha_1:=(c_1c_2) (P_1P_2)^2 \qquad \text{and} \qquad \alpha_2:= c_2 P_1P_2^2\,.$$
		The latter is a $2$-good basis because $P_1P_2$ and 
		$P_1P_2^2$ are $2$-independent modulo constants as they generate the same group as $P_1$,$P_2$. The $2$-divisibility parameters modulo constants are $D(\alpha_1)=1$ and $D(\alpha_2)=0$.
		We have $K(\sqrt{G})=K(\sqrt{c_1c_2}, \sqrt{c_2P_1})$. Thus, for any $n\geq 1$, we have
		$$
		\Gal(K(\zeta_{2^n},\sqrt[2^n]{G})/K(\zeta_{2^n})) \simeq
		\begin{cases}
			\mathbb Z/2^{n-1} \mathbb Z\times \mathbb Z/2^{n} \mathbb Z &\text{ if } c_1c_2\in \tk^{\times 2},\\
			\mathbb Z/2^{n} \mathbb Z\times \mathbb Z/2^{n} \mathbb Z &\text{ if } c_1c_2 \notin \tk^{\times 2}.
		\end{cases}
		$$
	\end{exa}
	
	\section*{Acknowledgment}
	
	The project started when the first author was supported by a PIMS postdoctoral fellowship at the University
	of Lethbridge and finished at the Universit\'e du Luxembourg where he was a postdoctoral researcher supported by the Fonds National de la Recherche, Luxembourg (17921905). We thank Jean Gillibert, Fritz Hörmann,  Qing Liu and Gabor Wiese for helpful discussions.

\end{document}